\documentclass[letterpaper,reqno,12pt]{amsart}
\usepackage[margin=1.2in]{geometry}
\usepackage{amssymb}
\usepackage{amsthm}
\usepackage{amsmath}
\usepackage{amsxtra}
\usepackage{latexsym}
\usepackage{mathrsfs}
\usepackage[all,cmtip]{xy}
\usepackage[all]{xy}
\usepackage{enumitem}
\usepackage{xcolor}
\usepackage{comment}
\usepackage{mathabx,epsfig}

\newtheorem{thm}{Theorem}[section]
\newtheorem{lem}[thm]{Lemma}

\newtheorem{prop}[thm]{Proposition}

\theoremstyle{definition}
\newtheorem{defn}[thm]{Definition}
\newtheorem{rmk}[thm]{Remark}
\newtheorem*{ack}{Acknowledgements}
\numberwithin{equation}{section}

\def\Z{{\mathbb Z}}

\def\sO{\mathcal{O}}
\def\m{\mathfrak{m}}

\DeclareMathOperator{\Pic}{Pic}
\DeclareMathOperator{\Cox}{Cox}

\title[]
{Characterization of toric varieties via int-amplified endomorphisms}
\author{Shou Yoshikawa}
\address{Graduate school of Mathematical Sciences, the University of Tokyo, Komaba, Tokyo,
153-8914, Japan}
\email{yoshikaw@ms.u-tokyo.ac.jp}

\begin{document}

\begin{abstract}
In this paper, we obtain a characterization of toric varieties via int-amplified endomorphisms.
We prove that if $f \colon X \to X$ is an int-amplified endomorphism of a smooth complex projective variety $X$, then $X$ is toric if and only if $f_*L$ is a direct sum of line bundles on $X$ for every line bundle $L$.
\end{abstract}

\maketitle

\setcounter{tocdepth}{1}

\section{Introduction}
Thomsen \cite{thomsen} proved that if $X$ is a smooth toric projective variety over an algebraically closed field of positive characteristic, the Frobenius push-forward $F_*L$ is a direct sum of line bundles on $X$  for every line bundle $L$.
Later, Achinger \cite{achinger} proved that this property indeed characterizes smooth toric varieties.
His proof depends on Kunz's theorem, which states that a Noetherian ring of prime characteristic is regular if and only if its Frobenius map is flat.
In this paper, we discuss a characteristic zero analog of this characterization via an int-amplified endomorphism instead of the Frobenius morphism.

A finite endomorphism $f \colon X \to X$ of a projective scheme $X$ is called \emph{int-amplified} if $f^*A \otimes A^{-1}$ is ample for some ample line bundle $A$ on $X$.
Multiplication maps on toric projective varieties are typical examples of int-amplified endomorphisms.
Note that every non-invertible endomorphism is int-amplified on a projective scheme with Picard rank one.
The following characterization gives a partial answer to the question of when admitting an int-amplified endomorphism implies being toric (cf. \cite{beauville}, \cite{hm03}, \cite{hwang-nakayama}, \cite{nakayama02}, \cite{ps89},\cite{meng-zhang-char}, \cite{meng-zhong}). 

\begin{thm}\textup{(Theorem \ref{only if part}, Theorem \ref{toric})}\label{main theorem}
Let $X$ be a smooth projective variety over an algebraically closed field $k$ of characteristic zero and $f$ be an int-amplified endomorpshim defined over $k$.
Then $X$ is toric if and only if $f_*L$ is a direct sum of line bundles on $X$ for every line bundle $L$.
\end{thm}

\begin{rmk}
If $f$ is a multiplication map, then the ``only if'' part of Theorem \ref{main theorem} essentially follows from the argument in \cite{thomsen}.
Toric varieties, however, have other endomorphisms in general. 
For example, a projective space has many endomorphisms which do not preserve any toric structure.
\end{rmk}

We briefly explain how to prove Theorem \ref{main theorem}.
In the proof of the ``if'' part, we first prove that
every \'etale cover $\widetilde{X}$ of $X$ equivariant under $f$ has the trivial Albanese variety.
It then follows from \cite[Theorem 1.3]{yoshikawa20fano} that $X$ is of Fano type, and in particular, the Picard group of $X$ is free and the Cox ring $R$ of $X$ is finitely generated by \cite{bchm}.
By an argument similar to \cite{achinger}, we show that the endomorphism $\phi$ on $R$ induced by $f$ is flat, using the property that $f_*L$ is a direct sum of line bundles on $X$ for every line bundle $L$.
A generalization of Kunz's theorem \cite[Theorem 13.3]{aim} tells us that $R$ is a polynominal ring and therefore $X$ is toric by \cite[Theorem 2.10]{hu-keel}.
Next we prove the ``only if'' part and we assume that $X$ is toric.
Then the Cox ring $R$ of $X$ is a polynominal ring and the induced endomorphism $\phi$ on $R$ is flat.
Given a line bundle $L$ on $X$, we consider the graded $R$-module $E_L$ defined by
\[
E_L := \bigoplus_{M \in \Pic(X)} H^0(X,f_*L \otimes M).
\]
$E_L$ is a direct summand of $\phi_*R$ by the projection formula, so  it is a free $R$-module.
Since $f_*L$ is the sheaf associated to $E_L$, the vector bundle $f_*L$ is a direct sum of line bundles by \cite[Proposition 6.A.3]{cox}.

Throughout this paper, all varieties are defined over an algebraically closed field $k$ of characteristic zero and all morphisms between varieties are $k$-morphisms. We will freely use the standard notations in \cite{kollar-mori}.

\begin{ack}
The author wishes to express his gratitude to his supervisor Professor Shunsuke Takagi for his encouragement, valuable advice and suggestions. He is also
grateful to Kenta Sato, Sho Ejiri, Teppei Takamatsu, Tatsuro Kawakami, Professor De-Qi Zhang and Professor Piotr Achinger for their helpful comments and suggestions. 
This work was supported by the Program for Leading Graduate Schools, MEXT, Japan.
He was also supported by JSPS KAKENHI Grant number JP20J11886.
\end{ack}

\section{Endomorphism on Cox rings}
In this section, we recall the definition of Cox rings, and study properties of endomorphisms of them induced by endomorphisms of varieties. Lemma \ref{characterization of regularity} gives a characterization of toric varieties via endomorphisms.

\begin{defn}
Let $X$ be a normal projective variety such that $\Pic(X)$ is free.
We take line bundles $L_1, \ldots L_r$ which are a basis of $\Pic(X)$.
We define the {\em Cox ring} of $X$ to be
\[
\Cox(X) := \bigoplus_{v \in \Z^r} H^0(X,L^r)
\]
with the multiplication induced by the product of rational functions, where
\[
L^v := L_1^{\otimes n_1} \otimes \cdots \otimes L_r^{\otimes n_r}
\]
for $v=(n_1, \ldots , n_r)$.
The Cox ring of $X$ is naturally graded by the Picard group of $X$ and every homogeneous part is a $k$-algebra, so we regard $\Cox(X)$ as a graded $k$-algebra. 
Furthermore, $\Cox(X)$ is independent on the choice of basis.
\end{defn}

\begin{prop}\label{canonical maximal ideal}
Let $X$ be a normal projective variety such that $\Pic(X)$ is free with basis $L_1, \ldots L_r$.
Then 
\[
\mathbf{m} := \bigoplus_{v \in \Z^r \backslash \{0\}} H^0(X,L^v)
\]
is a graded maximal ideal of $\Cox(X)$, it will be called the {\rm canonical maximal ideal} of $\Cox(X)$.
\end{prop}

\begin{proof}
First we prove that $\mathbf{m}$ is a graded ideal.
Taking homogeneous elements $\alpha \in \Cox(X)$ and $\beta \in \mathbf{m}$, then $\alpha$ and $\beta$ are global sections of some line bundles $L_\alpha$ and $L_\beta$, respectively.
If $\alpha \beta \in H^0(X,L_\alpha \otimes L_\beta)$ is not contained in $\m$, then we have $L_\alpha \otimes L_\beta \simeq \sO_X$.
If $\alpha$ and $\beta$ are non-zero global sections, then $L_\alpha$ and $L_\beta$ are trivial, and it contradicts to $\beta \in \m$.
Hence $\m$ is a graded ideal of $\Cox(X)$.
Furthermore, we have
\[
\Cox(X)/\m \simeq H^0(X,\sO_X) \simeq k,
\]
so $\m$ is a maximal ideal.
\end{proof}




\begin{lem}\label{polynominal}
Let $X$ be a normal projective variety such that $\Pic(X)$ is free.
Assume that $\Cox(X)$ is finitely generated $k$-algebra.
Then $X$ is smooth toric if and only if $\Cox(X)_\m$ is regular local ring, where $\Cox(X)_\m$ is the localization of the Cox ring by the canonical maximal ideal.
\end{lem}

\begin{proof}
By \cite[Theorem 2.10]{hu-keel}, $X$ is smooth toric if and only if $\Cox(X)$ is polynominal ring over $k$.
By \cite[Lemma 5]{achinger}, it equivalent to the regularity of $\Cox(X)_\m$.

\end{proof}

\begin{prop}\label{endomorphism of cox ring}
Let $X$ be a normal projective variety such that $\Pic(X)$ is free.
Let $f \colon X \to X$ be a finite morphism.
Then $f$ induces the injective $k$-algebra homomorphism $\phi \colon \Cox(X) \to \Cox(X)$ satisfying the following properties:
\\ \noindent{\rm (1)} $\phi$ is finite if $\Cox(X)$ is finitely generated $k$-algebra, and
\\ \noindent{\rm (2)} $\phi^{-1}(\m)=\m$ for the canonical maximal ideal $\m$.
\end{prop}

\begin{proof}
$f$ induces an injetive $k$-module homomorphism $H^0(X,L) \to H^0(X,f^*L)$ for any line bundle $L$.
Hence it induces the injective $k$-algebra homomorphism $\phi \colon \Cox(X) \to \Cox(X)$ such that the restriction of $\phi$ to $H^0(X,L)$ coincides with the above homomorphism.

By \cite[Theorem 3.1]{Okawa}, the submodule
\[
\bigoplus_{L \in \Pic(X)} H^0(X,f^*L)
\]
is a finite $\Cox(X)$-module.
Since $f^* \colon \Pic(X) \to \Pic(X)$ has the finite cokernel, we obtain the first assertion.

Next we prove the second assertion.
Let $s \in H^0(X,L)$ be a non-zero global section of a line bundle $L$. 
Then $\phi(s)$ is contained in $\m$ if and only if $f^*L$ is not trivial.
As $L$ has a non-zero global section, $f^*L$ is trivial if and only if $L$ is trivial.
Hence we have $\phi^{-1}(\m)=\m$.
\end{proof}

\begin{lem}\label{characterization of regularity}
Let $X$ be a normal projective variety such that $\Pic(X)$ is free.
Assume that $\Cox(X)$ is a finitely generated $k$-algebra.
Let $f \colon X \to X$ be a finite morphism.
Assume that if $f^*L \simeq L$ for a line bundle $L$, then $L$ is numerically trivial.
Then $X$ is smooth toric if and only if  the induced homomorphism $\phi \colon \Cox(X) \to \Cox(X)$ is flat.
\end{lem}

\begin{proof}
If $X$ is smooth toric, then $\Cox(X)$ is a polynominal ring.
Since $\phi \colon \Cox(X) \to \Cox(X)$ is a finite morphism of regular rings, $\phi$ is flat.

Next, we assume that $\phi$ is flat.
By Proposition \ref{endomorphism of cox ring}, $\phi$ is finite and induces local endomorphism of $\Cox(X)_{\mathfrak{m}}$.
By \cite[Theorem 13.3]{aim}, it is enough to show that $\phi$ is contracting, that is, $\phi^e(\m) \subset \m^2$ for some $e$.

We denote $\Cox(X)$ by $R$ and the degree $\lambda$ part by $R_{\lambda}$ for all $\lambda \in \Pic(X)$.
We take a homogeneous generator $\{a_i \in R_{\lambda_i}\}_{i=1, \ldots ,m}$ of $\m$.
We take a homogeneous element $a \in R_{\lambda} \subset \mathfrak{m}$, then we have $a=h(a_1, \ldots , a_m)$ for some polynominal $h$ with $h(0, \ldots , 0)=0$.
Since we have $\phi^l(a) =h(\phi^l(a_1), \ldots , \phi^l(a_m))$, it is enough to show that $\phi^l(a_i)$ is contained in $\mathfrak{m}^2$ for all $i$ for some $l$.
We assume that $\phi(a_i)$ is not contained in $R_{\lambda_j}$ for $j=1, \ldots m$.
Since we have $\phi(a_i)= h_i(a_1, \ldots ,a_m)$ for some polynominal $h_i$ and $\phi(a_i)$ is homogeneous,
the vanishing degree of $h_i$ is grater than two, and in particular, $\phi(a_i) \in \mathfrak{m}^2$.
Hence it is enough to show that for some $i$, there exists a positive integer $l$ such that $\phi^l(a_i)$ is not contained in $R_{\lambda_j}$ for all $j=1, \ldots m$.
Otherwise, for some $i$, $(f^{l})^*L_{\lambda_i}$ is isomorphic to $L_{\lambda_i}$ for some $l$, where $L_{\lambda_i}$ is a line bundle corresponding to $\lambda_i$.
By the assumption of $f$, $L_{\lambda_i}$ is numerically trivial.
Since we have $0 \neq a_i \in H^0(X,L_{\lambda_i})$, $L_{\lambda_i}$ is trivial, so it contradicts to $\lambda_i \neq 0$.
\end{proof}

\section{Proof of the ``only if'' part}
Let $X$ be a smooth toric projective variety. Thomsen \cite{thomsen} proved that for every line bundle $L$ on $X$, $F_*L$ is a direct sum of line bundles, where $F$ is the absolute Frobenius.
In this section, we prove an analogue of this result for all finite endomorphisms which is not necessarily a toric morphism.
\begin{thm}\label{only if part}\textup {(cf. \cite[Theorem 1]{thomsen})}
Let $X$ be a smooth toric projective variety and $f \colon X \to X$ a finite morphism.
Then for every line bundle $L$ on $X$, $f_*L$ is a direct sum of line bundles.
\end{thm}

\begin{proof}
Since $X$ is smooth toric, $R:=\Cox(X)$ is a polynominal ring with finite variables.
$f$ induces the endomorphism $\phi$ of $R$ as in Proposition \ref{endomorphism of cox ring}.
Since $\phi$ is a finite homomorphism of regular rings, $\phi$ is flat.
Thus $\phi_*R$ is flat and finite graded $R$-module.
Let $M$ be a line bundle on $X$.
We define a graded $R$-module $E_M$ by
\[
E_M := \bigoplus_{L \in \Pic(X)} H^0(X,f_*M \otimes L).
\]
Then we have a splitting injection
\[
E_M \simeq \phi_*\bigoplus_{L \in \Pic(X)} H^0(X,M \otimes f^*L) \to \phi_*R.
\]
Since $\phi_*R$ is flat, $E_M$ is a flat graded $R$-module, so $E_M$ is free $R$-module.
Hence we have $E_M \simeq R(\lambda_1) \oplus \cdots \oplus R(\lambda_d)$, where $\lambda_i \in \Pic(X)$ and $R(\lambda_i)$ is the shift of $R$.
By \cite[Proposition 6.A.3]{cox}, $f_*M$ is isomorphic to the sheaf associated to $E_M$, so we have
\[
f_*L \simeq L_{\lambda_1} \oplus \cdots \oplus L_{\lambda_d}.
\]

\end{proof}

\section{Proof of the ``if'' part}
In this section, we prove an analogue of Achinger's result for int-amplified endomorphisms, which are generalization of polarized endomorphisms.
It is the main difference from Achinger's proof that Lemma 4.5 and Lemma 4.6.

\begin{defn}
Let $X$ be a normal projective variety and $f \colon X \to X$ a finite endomorphism.
$f$ is called by \emph{int-amplified} if $f^*H-H$ is ample for some ample Cartier divisor $H$ on $X$.  
\end{defn}

\begin{lem}\label{numerically trivial}
Let $X$ be a normal projective variety admitting an int-amplified endomorphism $f$.
Let $L$ be a line bundle with $f^*L \simeq L$.
Then $L$ is numerically trivial.
\end{lem}

\begin{proof}
It follows from \cite[Theorem 3.3]{meng17}.
\end{proof}

\begin{lem}{\textup{(cf. \cite[Lemma2]{achinger})}}\label{lem:trivializable}
Let $X$ be a normal projective variety admitting an endomorphism $f$.
We assume that there exists a subgroup $\Lambda \subset \mathrm{Pic}(X)$ containing an ample line bundle such that for every line bundle $L \in \Lambda$, $f_*L \simeq L_1 \oplus \cdots \oplus L_d$ for some line bundles $L_1, \ldots L_d$ contained in $\Lambda$.
If a line bundle $M$ satisfies $f^*M \simeq \mathcal{O}_X$, then $M$ is trivial. 
\end{lem}

\begin{proof}
We take a line bundle $L \in \Lambda$.
Let $m(L',L)$ denote the multiplicity of a line bundle $L'$ as a direct summand of $f_*L$.
Let $M$ be a line bundle with $f^*M \simeq \mathcal{O}_X$.
We prove $m(L'\otimes M,L) \geq m(L',L)$.
We set $m:=m(L',L)$, than $L'^{\oplus m}$ is a direct summand of $f_*L$.
Hence, $(L'\otimes M)^{\oplus m}$ is a direct summand of $f_*L \otimes M$.
Since we have
\[
f_*L\otimes M \simeq f_*(L \otimes f^*M) \simeq f_*L,
\]
$(L'\otimes M)^{\oplus m}$ is a direct summand of $f_*L$, and in particular, we have $m(L'\otimes M,L) \geq m(L',L)$.
First we prove that $M$ is a torsion element of $\mathrm{Pic}(X)$.
By the assumption, there exists a line bundle $L'$ with $m(L',L) \geq 1$.
Hence, $L'\otimes M^k$ is a direct summand of $f_*L$ for all $k$, so we have $M^{k_1} \simeq M^{k_2}$ for some positive integers $k_1 \neq k_2$.
It means that $M$ is a torsion element of $\mathrm{Pic}(X)$.
Since $M$ is a torsion element, we obtain $m(L' \otimes M,L) = m(L',L)$ for all line bundles $L'$.
Let $k$ be the minimum positive integer of all positive integers with $M^k \simeq \mathcal{O}_X$.
Since we have $\chi(L'\otimes M)=\chi(L')$, we have
\[
\chi(L)=k(\chi(L_1)+ \cdots +\chi(L_m)),
\]
, in particular, $\chi(L)$ is divided by $k$.
Since $\chi(L_i)$ is also divided by $k$, $\chi(L)$ is divided by $k^2$.
Repeating such a process, $\chi(L)$ is divided by $k^e$ for all $L \in \Lambda$ and $e \in \Z_{\geq 0}$.
Since $\Lambda$ contains a very ample line bundle $L$ with $\chi(L) \neq 0$, we have $k=1$, and in particular, $M$ is trivial.
\end{proof}

\begin{lem}\label{lem:trivial albanese}
Let $X$ be a normal projective variety admitting an int-amplified endomorphism $f$.
We assume that there exists a subgroup $\Lambda \subset \mathrm{Pic}(X)$ containing an ample line bundle such that for every line bundle $L \in \Lambda$, $f_*L \simeq L_1 \oplus \cdots \oplus L_d$ for some line bundles $L_1, \ldots L_d$ contained in $\Lambda$.
Then the Albanese variety is a point.
\end{lem}

\begin{proof}
By \cite[Theorem 1.8]{meng17},
we have the following commutative diagram;
\[
\xymatrix{
X \ar[d]_f \ar[r]^a & A \ar[d]^g \\
X \ar[r]_a & A,
}
\]
where A is the albanese variety of $X$, $a$ is the albanese morphism of $X$ and $g$ is an int-amplified endomorphism.
By Lemma \ref{lem:trivializable}, for $f^* \colon \Pic^0(X) \to \Pic^0(X) $, the preimage of $0$ is a point, so $f^*$ is an isomorphism.
Since $g$ is a composition of a translation and the dual of $f^*$, $g$ is also an isomorphism.
In conclusion, $A$ is a point.
\end{proof}

\begin{lem}\label{lem:fiber product toric}
Let
\[
\xymatrix{
Y \ar[r]^{g} \ar[d]_{\pi} & Y \ar[d]^{\pi} \\
X \ar[r]_f & X
}
\]
be a commutative diagram of finite morphisms of smooth projective varieties such that $\pi$ is \'etale, $f$ and $g$ are  int-amplified endomorphisms.
We assume that there exists a subgroup $\Lambda \subset \mathrm{Pic}(X)$ containing an ample line bundle such that for every line bundle $L \in \Lambda$, $f_*L \simeq L_1 \oplus \cdots \oplus L_d$ for some line bundles $L_1, \ldots L_d$ contained in $\Lambda$.
Then the above diagram is a fiber product.
\end{lem}

\begin{proof}
By the \'etaleness, it is enough to show that $Y' := X \times_X Y$ is connected, and in particular, it is enough to show that $h^0(\pi^*f_*\sO_X)=1$.
We set $f_*\sO_X \simeq M_1 \oplus \cdots \oplus M_d $ for line bundles $M_1, \ldots M_d$.
We note that
\[
h^0(f^*M_i^{-1})=h^0(f_*\sO_X \otimes M_i^{-1}) \geq h^0(\sO_X) \geq 1.
\]
We may assume $h^0(M_1) =1$, then $h^0(f^*M_1^{-1}) \geq 1$ implies that $M_1$ is trivial.
Hence we have $\pi^*f_*\sO_X \simeq \sO_Y \oplus \pi^*M_2 \oplus \cdots \oplus \pi^*M_d $ and
it is enough to show $h^0(\pi^*M_i) = 0$ for all $i=2, \ldots , d$.

Suppose we have $h^0(\pi^*M_2) \geq 1$.
Since $h^0(\pi^*f^*M_2)$ and $h^0(\pi^*f^*M_2^{-1})$ is grater than zero, $\pi^*f^*M_2$ is trivial, and in particular, $f^*M_2$ is numerically trivial.
Since $h^0(f^*M_2^{-1})$ is grater that zero, $f^*M_2$ is trivial.
By Lemma \ref{lem:trivializable}, $M_2$ is also trivial, but it contradicts to $h^0(M_2)=0$.
\end{proof}

\begin{lem}\label{fano type}
Let $X$ be a smooth projective variety admitting an int-amplified endomorphism $f$.
If for every line bundle $L$ on $X$, $f_*L$ is a direct sum of line bundles, then $X$ is of Fano type, in particular, $\Pic(X)$ is free and $\Cox(X)$ is finitely generated.
\end{lem}

\begin{proof}
By \cite[Theorem1.3, 1.5]{yoshikawa20fano}, we have a commutative diagram
\[
\xymatrix{
Y \ar[r]^{g} \ar[d]_{\pi} & Y \ar[d]^{\pi} \\
X \ar[r]_f & X,
}
\]
where $\pi$ is an \'etale finite morphism, $g$ is an int-maplified endomorphism, and $Y$ is a smooth variety of Fano type over the Albanese variety of $Y$.
By Lemma \ref{lem:fiber product toric}, the above diagram is cartesian.
Let $\Lambda:=\pi^*\Pic(X)$ be a subgroup of $\Pic(Y)$, then it contains an ample line bundle.
Furthermore, for every $\pi^*L \in \Lambda$, we have
\[
g_*\pi^*L \simeq \pi^*f_*L.
\]
By the assumption, $f_*L$ is a direct sum of line bundles.
In particular, $g_*\pi^*L \simeq L_1 \oplus \cdots \oplus L_d$ for some $L_i \in \Lambda$.
By Lemma \ref{lem:trivial albanese}, the Albanese variety of $Y$ is a point, in particular, $Y$ is of Fano type.
Thus $X$ is smooth and rationally connected, so $\pi$ is trivial and $X$ is also of Fano type.
By \cite[Corollary 1.1.9]{bchm}, $\Cox(X)$ is finitely generated.
\end{proof}

\begin{thm}\label{toric}
Let $X$ be a smooth projective variety admitting an int-amplified endomorphism $f$.
If for every line bundle $L$ on $X$, $f_*L$ is a direct sum of line bundles, then $X$ is toric.
\end{thm}

\begin{proof}
By Lemma \ref{fano type}, $\Pic(X)$ is free and $R:=\Cox(X)$ is finitely generated.
By Lemma \ref{characterization of regularity} and Lemma \ref{numerically trivial}, it is enough to show that $\phi \colon R \to R$ induced by $f$ is flat.

For $\mu \in \Pic(X)/f^*\Pic(X)$, we take a representation $\mu' \in \Pic(X)$ of $\mu$.
Then we have
\[
\phi_*R \simeq \bigoplus_{\mu \in \Pic(X)/f^*\Pic(X)} \phi_*M_\mu
\]
as an $R$-module, where $M_\mu$ is defined by
\[
M_{\mu} := \bigoplus_{\lambda \in \Pic(X)} H^0(X,L_{\mu'}\otimes f^*L_{\lambda})
\]
By the projection formula, we have
\[
\phi_*M_{\mu}=\bigoplus_{\lambda \in \Pic(X)} H^0(X,f_*(L_{\mu'} \otimes f^*L_{\lambda}))=\bigoplus_{\lambda \in \Pic(X)} H^0(X,f_*L_{\mu'}\otimes L_{\lambda}).
\]
By the assumption, we have $f_*L_{\mu'} \simeq L_{\lambda_1} \oplus \cdots \oplus L_{\lambda_d}$.
Hence we have
\[
\phi_*M_{\mu} = \bigoplus_{i=1}^{d} \bigoplus_{\lambda \in \Pic(X)} H^0(X,L_{\lambda_i+\lambda}) \simeq \bigoplus_{i=1}^{d}R.
\]
It implies that $\phi$ is a flat finite homomorphism.
\end{proof}

\bibliography{bibliography}
\bibliographystyle{amsalpha}
\end{document}